\documentclass{amsart}

\usepackage{amssymb,amsmath,amsthm,bbm,enumitem}
\usepackage{hyperref}
\usepackage{}
\allowdisplaybreaks
\newtheorem{theorem}{Theorem}[section]
\newtheorem{corollary}[theorem]{Corollary}

\newtheorem{lemma}[theorem]{Lemma}
\newtheorem{assumption}[theorem]{Assumption}
\theoremstyle{remark}
\newtheorem{remark}{Remark}[section]

\newcommand{\ba}{\begin{enumerate}[label={\normalfont (\alph*)},ref={\normalfont \alph*}]}
\newcommand{\bei}{\begin{enumerate}[label={\normalfont (\roman*)},ref={\normalfont \roman*}]}
\newcommand{\be}{\begin{enumerate}[label={\normalfont (\arabic*)},ref={\normalfont \arabic*}]}
\newcommand{\ee}{\end{enumerate}}
\newcommand{\RR}{\ensuremath{\mathbb{R}}}
\newcommand{\Lp}{\left(}
\newcommand{\Rp}{\right)}
\newcommand{\LP}{\left\{ }
\newcommand{\RP}{\right\}}
\newcommand{\LT}{\left[}
\newcommand{\RT}{\right]}
\DeclareMathOperator{\Exp}{\mathbb{E}}
\DeclareMathOperator{\Prob}{\mathbb{P}}
\newcommand{\bds}{\boldsymbol}
\newcommand{\C}{\bds{C}}
\newcommand{\R}{\bds{R}}
\newcommand{\Chi}{\bds{\chi}}
\newcommand{\N}{\bds{N}}
\newcommand{\FF}{\mathcal{F}}
\newcommand{\indic}[1]{\mathbbm{1}_{\LT #1 \RT}}
\newcommand{\bone}{\mathbbm{1}}

\newcommand{\mynorm}[1]{\rho\Lp #1\Rp}
\newcommand{\Lone}[1]{\left\|#1\right\|}
\newcommand{\csd}[1]{\frac{\C_{#1}}{S_{#1}}}
\newcommand{\csl}[1]{\C_{#1}/S_{#1}}

\newcommand{\cslnmi}{C_{n-1,i}/S_{n-1}}
\newcommand{\mgdh}{\bds{\delta}}
\newcommand{\bxi}{\bds{\xi}}

\newcommand{\bH}{\bds{H}}
\newcommand{\lpf}{\bds{\pi}_{\bH}}
\newcommand{\pfval}{\lambda_{\bH}}
\newcommand{\bX}{\bds{X}}
\newcommand{\z}{\bds{z}}
\newcommand{\x}{\bds{x}}
\newcommand{\bgamma}{\bds{\gamma}}
\newcommand{\bGamma}{\bds{\Gamma}}

\newcommand{\CExp}[2]{\Exp\Lp #1\middle\vert\FF_{#2}\Rp}

\newcommand{\Tau}[2]{\tau_{#1}\Lp #2\Rp}
\newcommand{\0}{\bds{0}}
\newcommand{\1}{\bds{1}}
\newcommand{\hh}{h_{\bH}}
\newcommand{\e}{\bds{e}}
\newcommand{\X}{\bds{X}}
\newcommand{\Loc}{\bds{L}}
\newcommand{\tH}{\widetilde{\bH}}
\newcommand\numberthis{\addtocounter{equation}{1}\tag{\theequation}}

\begin{document}
\bibliographystyle{elsarticle-num}

\title[Randomized Urn Models and Elephant Random Walk]{Almost Sure Convergence of Randomized Urn Models\\ with Application to Elephant Random Walk}
\date{}
\author{Ujan Gangopadhyay}
\address{Department of Mathematics, \ National University of Singapore.} 
\email{ujan@nus.edu.sg}

\author{Krishanu Maulik}
\address{Theoretical Statistics and Mathematics Unit, \ Indian Statistical Institute, \ Kolkata, India.} 
\email{krishanu@isical.ac.in.}

\begin{abstract}
We consider a randomized urn model with objects of finitely many colors. The replacement matrices are random, and are conditionally independent of the color chosen given the past. Further, the conditional expectations of the replacement matrices are close to an almost surely irreducible matrix. We obtain almost sure and $L^1$ convergence of the configuration vector, the proportion vector and the count vector. We show that first moment is sufficient for i.i.d.\ replacement matrices independent of past color choices. This significantly improves the similar results for urn models obtained in \cite{AthreyaNey} requiring $L\log_+ L$ moments. For more general adaptive sequence of replacement matrices, a little more than $L\log_+ L$ condition is required. Similar results based on $L^1$ moment assumption alone has been considered independently and in parallel in \cite{ZhangNew}. Finally, using the result, we study a delayed elephant random walk on the nonnegative orthant in $d$ dimension with random memory.
\end{abstract}


\maketitle


\section{Introduction}

We consider an urn with objects of $K$ colors indexed by $[K]=\{1, 2, \ldots, K\}$. At discrete time points, we select a color with probability proportional to its content in the urn, and add more amount of object of each color to the urn following a randomized rule. For $n\geq 0$, consider the configuration vector $\C_n:=(C_{n1},\dots,C_{nK})$ where $C_{ni}$ is the amount of color $i$ in the urn at epoch $n$. Note that the amount added can be any nonnegative real value. Let $\Chi_n$ be the row vector indicating the color drawn at the $n$-th trial, so that $\Chi_n=\e_i$, the $i$-th coordinate vector in $\RR^K$, if the selected color is $i$. The randomized rule for adding content to the urn is encoded by a sequence of $K\times K$ random matrices $(\R_n)_{n\geq 0}$, called the replacement matrices. If $\Chi_n=\e_i$ then $R_{n;ij}$ amount of color $j$ is added to the urn, for each $j\in[K]$. Therefore, for $n\geq 1$
\begin{equation} \label{EvolutionEqn}
\C_n = \C_{n-1} + \Chi_n\R_n.
\end{equation}

The total content of the urn after $n$-th trial is $S_n = \sum_{i=1}^K C_{ni}$. The count vector after $n$-th trial, $\N_n$, giving the number of times each color is drawn, is $\N_n = \sum_{k=1}^n \Chi_k$. The vectors considered in this article are row vectors, and any column vector is indicated by transpose. We  also use $\1$ for a row vector of $1$'s, whose dimension is clear from the context. In this article we discuss sufficient conditions for almost sure convergence of $\C_n$, $S_n$, $\N_n$ after suitable scaling.

For $n\geq 1$, let $\FF_{n}$ be the $\sigma$-field generated by $\C_0$ and $\Lp\R_m, \Chi_m\Rp_{m=1}^{n}$. The basic assumptions about the evolution of the urn are the following.

\begin{assumption} \label{Assmp: basic}
The adapted sequence $((\Chi_n, \R_n), \FF_n)_{n\geq 1}$ satisfies:
\bei
    \item The nonzero vector $\C_0$ has almost surely nonnegative entries and finite mean. \label{Assmp: init}
    \item For $n\geq 1$ and $i\in[K]$,
    $\Prob\Lp\Chi_n=\bds{e}_i|\FF_{n-1}\Rp=\cslnmi.$
    \item Each of the replacement matrices $(\R_n)_{n\geq 1}$ have finite mean.
    \item The variables $\Chi_n$ and $\R_n$ are conditionally independent given $\FF_{n-1}$.
\ee
\end{assumption}

We also assume that the replacement matrices are ``close'' to an irreducible replacement matrix in an appropriate sense. Recall that, a square matrix $\bds{A}$ with nonnegative entries is called \textit{irreducible} if, for any $i, j \in [K]$, there exists a positive integer $N\equiv N_{ij}$, such that the $(i,j)$-th entry of $\bds{A}^N$ is positive. By Perron-Frobenius theory, the \textit{dominant eigenvalue} of an irreducible matrix, namely the eigenvalue having the largest real part, is simple and positive and has corresponding left and right eigenvectors with all coordinates strictly positive.

For a matrix $\bds{A}$, we use the operator norm $\mynorm{\bds{A}} := \max_i\sum_{j}\left|A_{ij}\right|$.
Further, for an event $A$, its indicator function is denoted by $\bone_A$.
The conditional expectations of the replacement matrices, denoted as $\bH_{n-1} := \CExp{\R_n}{n-1}$, are called the \textit{generating matrices} and their truncated version, denoted as $\tH_{n-1} := \CExp{\R_n \indic{\mynorm{\R_n}\le n}}{n-1}$, are called \textit{truncated generating matrices}.

\begin{assumption}\label{Assmp: ae}
There exists a (possibly random) almost surely irreducible matrix $\bH$ with nonnegative entries, such that the truncated generating matrices $\tH_n$ converge to $\bH$ in Cesaro sense in the operator norm almost surely, i.e., $\frac{1}{n}\sum_{k=0}^{n-1} \mynorm{\tH_{k}-\bH}\to 0$. The dominant eigenvalue of $\bH$ is $\pfval$ and $\lpf$ is the unique left eigenvector of $\bH$ corresponding to $\pfval$ such that $\lpf$ has all coordinates strictly positive and is normalized to be a probability vector.
\end{assumption}

In \cite{AthreyaNey}, the replacement matrices formed an i.i.d.\ sequence, with finite $L \log_+ L$ moments for the entries and irreducible mean matrix, which was additionally assumed to be aperiodic. Using branching process techniques, it was shown in \cite{AthreyaNey}, Chapter~V.9.3 that the normalized color count and composition vectors converge almost surely to the (normalized to probability) left eigenvector of the mean matrix corresponding to its dominant eigenvalue. In this article, we establish almost sure convergence for i.i.d.\ replacement matrices under $L^1$ conditions alone, as a significant relaxation of $L \log_+ L$ condition in \cite{AthreyaNey}. The mean replacement matrix is taken to be irreducible alone, aperiodicity is not needed. We extend the result to an adapted sequence also, under  $L \log_+ L$ condition, when the replacement matrices are suitably majorized. A referee has pointed out an unpublished manuscript~\cite{ZhangNew}. It became available on the preprint server arXiv long after the initial submission of the present work. The manuscript~\cite{ZhangNew} obtained similar results independently and in parallel, using different stochastic approximation method. The stochastic approximation used here enables one to obtain a simpler proof where the associated differential equation is simpler. In the process, we develop Theorem~\ref{Thm: SA} on Stochastic Approximation which uses only Cesaro negligibility of the error sequence and can be of independent interest.

In Section~\ref{sec: lv}, we state the model and the main results. Theorem~\ref{thm: iid} extends the result for i.i.d.\ replacement matrices in \cite{AthreyaNey, Zhang12} under $L^1$ condition only. Corollary~\ref{cor: L log L} studies adapted replacement matrices, majorized by a $L \log_+ L$ random variable.
In Section~\ref{sec:urn}, the main tool for the proof -- an appropriate result on stochastic approximation -- is provided and the main results are proven. Finally, in Section~\ref{sec:erw}, we apply the results for the urn to study a delayed elephant random walk on $d$-dimensional nonnegative orthant with randomly reinforced memory.

\section{Main Results} \label{sec: lv}

To obtain the appropriate convergence results, we make further uniform integrability type conditions on the replacement matrices.

\begin{assumption} \label{Assmp: unif tail}
The sequence $(\mynorm{\R_n})$ satisfy one of the following:
\ba
\item \label{Assmp: maj} The distributions of $\mynorm{\R_n}$ are majorized: there exists $c\in (0,\infty)$ and a positive random variable $R$ with finite expectation such that, for all $x>0$, $$\Prob\Lp\mynorm{\R_n}>x\Rp \le c \Prob(R>x).$$

\item \label{Assmp: alt tail} For some nonnegative function $\phi$ on $[0, \infty)$ which is eventually positive and nondecreasing, satisfying $\sum 1/(n\phi(n))< \infty$ and $x/\phi(x)$ eventually monotone nondecreasing, we have the bounded moment condition: $$\sup_n \Exp(\mynorm{\R_n}\phi(\mynorm{\R_n}) < \infty.$$
\ee
\end{assumption}

\begin{remark} 
An example of $\phi(x) = (\log_+ x)^p$ for some $p>1$, has been considered by~\cite{Zhang12}. Other choices include $\phi(x) = \log_+x (\log_+ \log_+ x)^p$ for some $p>1$.
\end{remark}

\begin{remark}
It is interesting to note that the majorizing condition in Assumption~\ref{Assmp: unif tail}\eqref{Assmp: alt tail} is on the unconditional distribution of the replacement matrices $(\R_n)$, in contrast to the conditions~(3.3)--(3.5) of Theorem~3.1 of~\cite{ZhangNew}.
\end{remark}

\begin{remark} \label{Rem: L1}
Both Assumptions~\ref{Assmp: unif tail}\eqref{Assmp: maj} and~\ref{Assmp: unif tail}\eqref{Assmp: alt tail} imply uniform integrability of  $\Lp\mynorm{\R_n}\Rp$ and, hence, of $\Lp\R_n\indic{\mynorm{\R_n}\le n}\Rp$. Thus, under either condition, the convergence in Assumption~\ref{Assmp: ae} is also in $L^1$.
\end{remark}

\begin{theorem} \label{Main}
Under Assumptions~{\normalfont \ref{Assmp: basic}},~{\normalfont \ref{Assmp: ae}} and~{\normalfont \ref{Assmp: unif tail}}, we have, almost surely, as well as in $L^1$,
\begin{equation}\label{strong conv}
  \frac{\C_n}{S_n} \to \bds{\pi}_{\bH}; \quad \frac{S_n}n \to \lambda_{\bH}; \quad
  \frac{\C_n}n \to \lambda_{\bH} \bds{\pi}_{\bH}; \quad \text{and} \quad
  \frac{\N_n}n \to \bds{\pi}_{\bH}.
\end{equation}
\end{theorem}

\begin{remark}
It is interesting to compare the assumptions in this article with those required to establish convergence in probability and in $L^1$ in \cite{GangoMaulik}. Note that Assumption~\ref{Assmp: basic} is same as Assumption~3.1 of \cite{GangoMaulik}. Assumption~\ref{Assmp: ae} corresponds to Assumptions~3.2 and~3.3 of \cite{GangoMaulik}. Assumption~3.2 of \cite{GangoMaulik} on the properties of the matrix $\bH$ remain unchanged. However, given the almost sure convergence in conclusion, we naturally strengthen the assumption to almost sure convergence. The methods of proof differ however. We further require the truncated generating matrices $\tH_n$ to converge to $\bH$ instead of the generating matrices, as considered in \cite{GangoMaulik}. In Remark~\ref{Rem: tail ce} sufficient conditions will be provided for convergence of $\bH_n$, as in \cite{GangoMaulik}, to work. Finally, uniform integrability condition in Assumption~3.4 of \cite{GangoMaulik} has been appropriately strengthened in Assumption~\ref{Assmp: unif tail}.
\end{remark}

The case of i.i.d.\ replacement matrices have been considered in \cite{AthreyaNey, Zhang12}. Assumption~\ref{Assmp: unif tail}\eqref{Assmp: maj} holds, when $(\R_n)$ is i.i.d.\ with finite mean. Further, since $R_n$ is independent of $\FF_{n-1}$ and i.i.d.\ with finite mean, Assumption~\ref{Assmp: ae} holds. Thus, we relax $L \log_+ L$ condition of \cite{AthreyaNey, Zhang12} to existence of first moment alone and obtain the following theorem as a corollary to Theorem~\ref{Main}. As has been noted before, a referee has informed us about the unpublished manuscript~\cite{ZhangNew}, which was developed independently and in parallel, and considered the following theorem in its Corollary~2.1. However, a different stochastic approximation has been used here, which has significantly simplified the proof.

\begin{theorem} \label{thm: iid}
We consider an urn model with nonzero initial configuration $\C_0$ having almost surely nonnegative entries and finite mean; as well as an i.i.d.\ sequence of replacement matrices $(\R_n)$, independent of $\C_0$ with irreducible mean matrix and $\Exp\Lp\mynorm{\R_1}\Rp<\infty$. Also, let
\[
\Prob\Lp\Chi_n=\bds{e}_i|\C_0,(\R_m)_{m=1}^{n-1}, (\Chi_m)_{m=1}^{n-1}\Rp=\cslnmi
\]
hold for $n\geq 1$ and $i\in[K]$, and let $\Chi_n$ be independent of $(\R_m)_{m\ge n}$. Then the convergence~\eqref{strong conv} in Theorem~{\normalfont \ref{Main}} holds both almost surely and in $L^1$.
\end{theorem}

It is natural to also consider the following variant of Assumption~\ref{Assmp: ae}.

\begin{assumption}\label{Assmp: aeprime}
There exists a (possibly random) matrix $\bH$ with nonnegative entries, which is almost surely irreducible, such that the generating matrices $\bH_n$ converge to $\bH$ in Cesaro sense in the operator norm almost surely i.e.,
$\frac1n \sum_{k=0}^{n-1} \mynorm{\bH_{k}-\bH}\to 0$ almost surely.
\end{assumption}

\begin{remark} \label{Rem: tail ce}
By Kronecker's lemma, Assumption~\ref{Assmp: aeprime} and the condition
\begin{equation}
\sum_{n=1}^{\infty}\frac{1}{n}\CExp{\mynorm{\R_{n}}\indic{\mynorm{\R_{n}}>n}}{n-1}<\infty\text{ almost surely,} \label{eq: tail ce neg}
\end{equation}
imply Assumption~\ref{Assmp: ae}.
\end{remark}

\begin{remark}
It may be noted that the analysis in~\cite{ZhangNew} is done under its assumption~(3.1) alone, which is same as Assumption~\ref{Assmp: aeprime}. The additional truncated moment condition~\eqref{eq: tail ce neg} corresponds to~(3.4) in~\cite{ZhangNew}.
\end{remark}

Corollary~\ref{cor: L log L} extends the result on i.i.d.\ sequence of replacement matrices given in \cite{AthreyaNey} to adapted sequence with $L \log_+ L$ moment for the majorizing random variable. The moment condition on the majorizing random variable in Assumption~\ref{Assmp: unif tail}\eqref{Assmp: maj} is strengthened to replace Assumption~\ref{Assmp: ae} by Assumption~\ref{Assmp: aeprime}.

\begin{corollary} \label{cor: L log L}
Consider an adapted sequence $\Lp \Lp \Chi_n, \R_n\Rp, \FF_n\Rp$ satisfying Assumptions~{\normalfont\ref{Assmp: basic}},~{\normalfont\ref{Assmp: unif tail}\eqref{Assmp: maj}} and~{\normalfont\ref{Assmp: aeprime}}, where the majorizing random variable $R$ in Assumption~{\normalfont\ref{Assmp: unif tail}\eqref{Assmp: maj}} additionally satisfies $\Exp(R \log_+ R) < \infty$. Then the convergence~\eqref{strong conv} in Theorem~{\normalfont \ref{Main}} holds both almost surely and in $L^1$.
\end{corollary}

\begin{remark}
It may be noted that Assumption~\ref{Assmp: unif tail}\eqref{Assmp: maj} involves majorization of the unconditional distribution of the replacement matrices in contrast to the conditional distribution in~\cite{ZhangNew}.
\end{remark}

Next corollary uses a generalization of Assumption~2.2(a) in~\cite{Zhang12}. See also Remark~3.1 of~\cite{ZhangNew}.
No additional condition is required to replace Assumption~\ref{Assmp: ae} by Assumption~\ref{Assmp: aeprime}.

\begin{corollary} \label{cor: phi}
Assumptions~{\normalfont \ref{Assmp: basic}},~{\normalfont \ref{Assmp: unif tail}\eqref{Assmp: alt tail}} and~{\normalfont \ref{Assmp: aeprime}} imply the convergence~\eqref{strong conv} in Theorem~{\normalfont \ref{Main}} both almost surely and in $L^1$.
\end{corollary}

\section{Proofs of the Main Results} \label{sec:urn}
In the first subsection, we provide the relevant result on stochastic approximation and another technical result. In the last two subsections, we prove the main theorem and the corollaries of Section~\ref{sec: lv}.

\subsection{Stochastic Approximation and Other Results}
The main tool for proving Theorem~\ref{Main} is Stochastic Approximation. We shall use the following theorem on Stochastic Approximation, which follows from a special case of Theorem~9.2.8 of \cite{Duflo}, which in turn was adapted from \cite{KushnerClark}.

\begin{theorem} \label{Thm: SA}
    Let $S$ be a compact and convex subset of $\RR^K$. Let $h:S\to\RR^K$ be
    a continuous function. Let $\Lp a_n\Rp_{n=0}^\infty$ be a nonincreasing
    sequence of positive real numbers, called the step sizes, satisfying
    \begin{equation} \label{eq: step order}
        0<\liminf {n}{a_n} \le \limsup {n}{a_n} <\infty.
    \end{equation}
    Let $\Lp\bgamma_n\Rp_{n=0}^\infty$ be a sequence, called the error
    sequence, which is Cesaro negligible, i.e., $\frac1n\bGamma_n\to\0$,
    where $\bGamma_n := \sum_{i=0}^n \bgamma_i$. Let
    $\Lp\x_n\Rp_{n=0}^\infty$ be a sequence that takes values in $S$ and
    evolves as
    \begin{equation}\label{eq: SA}
        \x_{n+1}=\x_n+a_n h(\x_n)+a_n\bgamma_n.
    \end{equation}
    If the ODE $\dot{\x}=h(\x)$ has a unique solution $\z:\RR\to S$, then
    $\lim_{n\to\infty}\x_n=\z(0)$.
\end{theorem}

The above results uses only Cesaro negligibility of the error sequence, rather than its summability or negligibility. The result can be of independent interest.

It should also be noted that general form of step sizes allows us to develop stochastic approximation for $(\C_n/S_n)$ in contrast to $(C_n/n)$ considered in~\cite{ZhangNew}. As $(\C_n/S_n)$ already lies in the probability simplex and the corresponding differential equation is of the standard Lotka-Volterra type, our analysis becomes significantly simpler.

\begin{proof}[Proof of Theorem~\ref{Thm: SA}]
It is immediate from~\eqref{eq: step order} that $a_n\to 0$ and $\sum_{n=0}^\infty a_n=\infty$. Then, using Theorem~9.2.8 of \cite{Duflo}, it is enough to show that, for all $T>0$,
\[
\lim_{n\to\infty}\sup_{n\leq j\leq\Tau{n}{T}}\Lone{\sum_{k=n}^j a_k\bgamma_k}=0,
\]
where $\Tau{n}{T}:=\inf\LP k\geq n: a_n+\cdots+a_k\geq T\RP,$ and $\Lone{\cdot}$ is any norm in $\RR^K$.
Fix $T>0$. First observe that
\[
a_k\bgamma_k = \Lp a_k\bGamma_k - a_{k-1}\bGamma_{k-1}\Rp +
a_k a_{k-1}\bGamma_{k-1} \Lp \frac1{a_k} - \frac1{a_{k-1}} \Rp.
\]
Since $(a_n)$ is nonincreasing, for $n\leq j\leq\Tau{n}{T}$, we have,
\[
\Lone{\sum_{k=n}^j a_k\bgamma_k} \le \Lone{a_j\bGamma_j} + \Lone{a_{n-1}\bGamma_{n-1}} + a_n \sup_{k\ge n-1} \Lone{a_k \bGamma_k} \frac{1}{a_j}.
\]
Hence, we have
\[
\sup_{n\leq j\leq\Tau{n}{T}}\Lone{\sum_{k=n}^j a_k\bgamma_k} \le \sup_{j\ge n-1} \Lone{a_j\bGamma_j} \Lp 2 + na_n \cdot \frac{\Tau{n}{T}}{n} \cdot \frac1{\Tau{n}{T}a_{\Tau{n}{T}}}\Rp.
\]
Using~\eqref{eq: step order} and $\frac{1}{n}\bGamma_n\to 0$, it is enough to show that $\Tau{n}{T}/n$ is bounded for any $T>0$. Since $\liminf na_n >0$, choose $a^*>0$ such that $n a_n > a^*$ for large enough $n$. Also $a_n < T$ for large enough $n$. Hence, for large enough $n$, $\Tau{n}{T}>n$ and
\[
a^*\left[\log \frac{\Tau{n}{T}-1}{n-1} -1\right] \le \frac{a^*}{n} + \cdots + \frac{a^*}{\Tau{n}{T}-1} \le a_n + \ldots + a_{\Tau{n}{T}-1} \le T.
\]
Thus, $\Tau{n}{T}/n \le (\Tau{n}{T}-1)/(n-1) \le \exp(T/a^*+1)$, as required.
\end{proof}

We also need the following technical result for the average of a dominated adapted sequence, with the terms appropriately centered.

\begin{lemma} \label{lemma: trunc mg}
Let $(X_n, \FF_n)$ be adapted, and $|X_n|\le V_n$, where $(V_n)$ satisfies
\begin{itemize}
\item[-] either a majorization condition as in Assumption~{\normalfont \ref{Assmp: unif tail}\eqref{Assmp: maj}}, namely there exists $c'\in(0,\infty)$ and a positive random variable $V$ with finite expectation satisfying $\Prob(V_n>x) \le c'\Prob(V>x)$ for all $x>0$,
\item[-] or a bounded moment condition as in Assumption~{\normalfont \ref{Assmp: unif tail}\eqref{Assmp: alt tail}}, namely a function $\phi$ as in Assumption~{\normalfont \ref{Assmp: unif tail}\eqref{Assmp: alt tail}}, satisfying $\sup_n \Exp(V_n \phi(V_n)) <\infty$,
\end{itemize}
then, $\frac1n \sum_{k=1}^n \Lp X_k - \CExp{X_k \indic {V_k\le k}}{k-1} \Rp \to 0$ almost surely.

Further, if $(V_n)$ is identically distributed with finite mean and if $V_n$ is independent of $\FF_{n-1}$ for all $n$, then  $\frac1n \sum_{k=1}^n \Lp X_k - \CExp{X_k}{k-1} \Rp \to 0$ almost surely.
\end{lemma}

\begin{proof}
The first part of the proof, where centering is done by the truncated moment, is similar to the derivation of~{(2.20)} of Theorem~2.19 of \cite{HH} under majorization condition. Under bounded moment condition, we additionally use, for all large enough $n$, $\Prob(V_n>n) \le \frac1{n \phi(n)} \Exp(V_n \phi(V_n))$ and $\CExp{V_n^2 \indic{V_n\le n}}{n-1} \le \frac{n}{\phi(n)}\CExp{V_n\phi(V_n)}{n-1}$.

Finally, when $(V_n)$ is identically distributed with finite mean and $V_n$ is independent of $\FF_{n-1}$, the majorization condition holds. Also
\[
\CExp{X_n \indic{V_n>n}}{n-1}\le \Exp{V_n \indic{V_n| > n}} \to 0
\qquad\qquad\mbox{almost surely,}
\]
by independence of $V_n$ from $\FF_{n-1}$, which takes care of the remaining term in the centering.
\end{proof}

\subsection{Proof of Theorem~\ref{Main}} 
First we establish convergence of $(\C_n/S_n)$ to $\lpf$ by rewriting the
evolution equation~\eqref{EvolutionEqn} in the form of \eqref{eq: SA} and checking the conditions of Theorem~\ref{Thm: SA}. Define $Y_0:=S_0$ and for $n\geq 1$, $Y_n:=S_n-S_{n-1}=\Chi_n\R_n\bds{1}^T$. So $Y_n$ is the total amount added to the urn after the $n$-th draw. Observe that, for $n\geq 1$,
\begin{equation}
  0\leq Y_n=\Chi_n\R_n\bds{1}^T\leq\mynorm{\R_n}. \label{Y bd}
\end{equation}
We rewrite the evolution equation \eqref{EvolutionEqn} as
\[
\csd{n} = \csd{n-1} + \frac{1}{S_n}\hh\Lp\csd{n-1}\Rp + \frac{\mgdh_n}{S_n} + \frac{\bxi_n}{S_n},
\]
where $\hh$, $\mgdh_n$, $\bxi_n$ are defined as follows. The drift $\hh$, indexed by $K\times K$ matrices $\bH$, is defined as $\hh\Lp\bX\Rp := \bX\bH - \bX\Lp\bX\bH\bds{1}^T\Rp$. For each $n\geq 1$,
\[
\mgdh_n := \Lp\Chi_n\R_n-\csd{n-1}Y_n\Rp - \CExp{\Lp\Chi_n\R_n-\csd{n-1}Y_n\Rp \indic{\mynorm{\R_n}\le n}}{n-1},
\]
is the martingale difference term, and $\bxi_n := h_{\tH_{n-1}-\bH}\Lp\csd{n-1}\Rp$ is the adjusted truncated conditional expectation term.
The $\Lp\csl{n}\Rp$ takes values in the closed bounded convex set of probability vectors in $\RR^K$. The corresponding differential equation is $\dot{\x}=\hh\Lp\x\Rp$ with $\hh(\x)$, a quadratic polynomial in $\x$, is continuous. Further, from Proposition~3.3 of~\cite{GangoMaulik}, the differential equation has unique solution in the probability simplex given by $\x(t)=\lpf$ for all $t$. Now we check the conditions on the step sizes. Since $S_n$ is partial sum of nonnegative $(Y_n)$, the step sizes are nonincreasing. Now we check~\eqref{eq: step order}. Define
\begin{multline*}
\eta_n := \frac{S_n}{n}-\frac{1}{n}\sum_{m=1}^n\frac{\C_{m-1}}{S_{m-1}}\bH\1^T\\
= \LT\frac{S_n}n - \frac1n \sum_{m=1}^n \CExp{Y_m \indic {\mynorm{\R_m}\le m}}{m-1}\RT + \frac1n \sum_{m=1}^n \csd{m-1}\Lp\tH_{m-1}-\bH\Rp\bds{1}^T.
\end{multline*}
Using the bound~\eqref{Y bd} and Lemma~\ref{lemma: trunc mg}, $\frac{S_n}n - \frac1n \sum_{m=1}^n \CExp{Y_m \indic {\mynorm{\R_m}\le m}}{m-1}\to 0$ almost surely. Also, $\frac1n \sum_{m=1}^n \csd{m-1}\Lp\tH_{m-1}-\bH\Rp\bds{1}^T\to 0$, using Assumption~\ref{Assmp: ae}. Thus $\eta_n\to 0$ almost surely.
Let $\sigma(\bH)$ be the least absolute row sum of $\bH$. Since $\bH$ is irreducible, no row can be the zero vector and hence $\sigma(\bH)>0$. Using $\eta_n\to 0$ and
$
0 < \sigma\Lp\bH\Rp \leq \frac{1}{n}\sum_{m=1}^n\frac{\C_{m-1}}{S_{m-1}}\bH\1^T \leq \mynorm{\bH} < \infty
$
we get that the step sizes $(1/S_n)$ satisfy condition~\eqref{eq: step order}. Finally, each coordinate of $(\mgdh_n)$ is Cesaro negligible using Lemma~\ref{lemma: trunc mg} and the bound $\Lone{\Chi_n\R_n-\csd{n-1}Y_n}\leq 2\mynorm{\R_n}$. Also, $(\bxi_n)$ is Cesaro negligible by Assumption~\ref{Assmp: ae}.

Then from Theorem~\ref{Thm: SA} we have the almost sure convergence of the proportion vector $\C_n/S_n$ to $\bds{\pi}_{\bH}$.
Using $\C_n/S_n\to\lpf$ we get
$S_n/n-\eta_n=\frac{1}{n}\sum_{m=1}^n\frac{\C_{m-1}}{S_{m-1}}\bH\1^T\to\lpf\bH\1^T=\pfval$. Since $\eta_n\to 0$, we get $S_n/n\to\pfval$, and hence $\C_n/n \to \pfval \bds{\pi}_{\bH}$ almost surely.
Now
$
\frac{\N_n}{n} - \frac{1}{n}\sum_{m=0}^{n-1}\csd{m} = \frac{1}{n}\sum_{m=1}^n\Lp\Chi_m-\csd{m-1}\Rp,
$
being a scaled $L^2$-bounded martingale, is negligible. The almost sure convergence of $\N_n/n$ to $\lpf$ then follows from that of $\C_n/S_n$.

Being bounded, $\csl{n}$ and $\N_n/n$ converge in $L^1$. We next consider $S_n/n$.
From~\eqref{Y bd}, $0\leq Y_n\leq\mynorm{\R_n}$ for $n\ge 1$ and from Assumption~\ref{Assmp: basic}~(\ref{Assmp: init}) $Y_0= \C_0 \1^T$ is integrable. As noted in Remark~\ref{Rem: L1}, Assumption~\ref{Assmp: unif tail} implies uniform integrability of $(\mynorm{\R_n})$. Hence $\Lp Y_n\Rp$ and $\Lp S_n/n\Rp$ are also uniformly integrable. Hence $S_n/n$ converges $L^1$. Also, so does $\C_n/n$ using Lemma~3.5 of \cite{GangoMaulik}.

\subsection{Proofs of Corollaries~\ref{cor: L log L} and~\ref{cor: phi}}

Assumption~\ref{Assmp: unif tail}\eqref{Assmp: maj} with $L \log_+ L$ moment condition gives
\begin{align*}
&\sum_{n=1}^{\infty} \frac1n \Exp\Lp\mynorm{\R_{n}}\indic{\mynorm{\R_{n}}>n}\Rp \le c \sum_{n=1}^{\infty} \sum_{j=n}^\infty \frac1n \Exp\Lp R\indic{R\in (j, j+1]}\Rp\\
&\qquad \qquad \le c \sum_{j=1}^{\infty} \Exp\Lp R\indic{R\in (j, j+1]}\Rp \sum_{n=1}^j \frac1n \le c \Exp\Lp R (1+ \log_+ R)\Rp < \infty.
\end{align*}
This implies~\eqref{eq: tail ce neg} and, using Remark~\ref{Rem: tail ce} and Theorem~\ref{Main}, proves Corollary~\ref{cor: L log L}.

Similarly, Assumption~\ref{Assmp: unif tail}\eqref{Assmp: alt tail} gives, for large enough $m$
\[
\sum_{n=m}^\infty \frac1n \Exp\Lp\mynorm{\R_{n}}\indic{\mynorm{\R_{n}}>n}\Rp \le \sum_{n=m}^\infty \frac1{n \phi(n)} \Exp\Lp\mynorm{\R_{n}}\phi\Lp\mynorm{\R_{n}}\Rp\Rp<\infty.
\]
This again implies~\eqref{eq: tail ce neg} and proves Corollary~\ref{cor: phi}.

\section{Applications to Elephant Random Walk}\label{sec:erw}

We consider a delayed elephant random walk (ERW) on nonnegative integer lattice of dimension $d$ with randomly reinforced memory. ERW was introduced in \cite{ERWoriginal}. We analyze the model using the results obtained in this article utilizing an interesting connection between the urn model and ERW discovered in \cite{Bertoin}.

The random walk is parametrized by three parameters, namely, mean memory reinforcement parameter $a>0$ and two mixing parameters $p, q\in(0,1)$ for delay and shift respectively. At every epoch, a past epoch is selected with probability proportional to its memory and the memory of the selected epoch is randomly reinforced. If there was a delay at the selected epoch, then the current epoch is also delayed with probability $1-p$ or else there is a unit movement in a randomly selected direction. If there was a shift at the selected epoch, it is repeated with probability $1-q$ or else there is a delay.

To construct the random walk, we consider three mutually independent i.i.d.\ sequences $(\bds{A}_n)_{n>1} := ((A_{0n},A_{1n},\dots,A_{dn}))_{n\geq 1}$, $(I_n)_{n\geq 1}$ and $(J_n)_{n\geq 1}$. The memory reinforcement $\bds{A}_n$ has nonnegative coordinates. The coordinates have common finite first moment $a$, but may have different distributions. The variable $I_n$ takes values $0, 1, \ldots, d$ with probabilities $1-p, p/d, \ldots, p/d$ respectively, while $J_n$ is a Bernoulli ($1-q$) random variable. Independent of these sequences, $\bds{U}$ is a random vector, uniform over $\LP\e_0,\dots,\e_d\RP$, where $\e_0$ is the zero vector in $\RR^d$ and, as before, for each $i\geq 1$, $\e_i$ is the $i$-th coordinate vector in $\RR^d$.

At epoch $n$, the elephant is at $\Loc_n$ and the step taken is $\X_n$, giving, $\Loc_n = \Loc_{n-1} + {\X_n}$. The step $\X_n$ takes values $\e_0, \e_1, \ldots, \e_d,$ corresponding to a delay or a unit step shift along one of the $d$ coordinate axes respectively. The elephant starts at the origin, i.e., $\Loc_0 = \e_0$, and the first step $\X_1$ is taken to be $\bds{U}$.

For $n\geq 1$, let $\mathcal F_n$ denote the $\sigma$-field generated by $\bds{U}$, $(I_m)_{m\leq n}$, $(J_m)_{m\leq n}$. An adapted sequence $((M_1^{(n)}, \dots, M_{n}^{(n)}), \FF_n)_{n\geq 1}$ denotes the memory of the elephant about the past epochs evolving over time. The memory sequence is initiated by taking $M_1^{(1)}=1$. At epoch $n>1$, the elephant chooses $\tau_n$, one of the past epochs, such that,
\[
\Prob\Lp\tau_n = u | \mathcal F_{n-1}\Rp = \frac{M_u^{(n-1)}}{ M_1^{(n-1)}+\cdots+M_{n-1}^{(n-1)}}, \mbox{ for } u = 1, \ldots, n-1.
\]
If the step $\X_{\tau_n}$, at the selected epoch $\tau_n$, was $\e_i$ for some $i = 0, 1, \ldots, d$, then the memory $M_{\tau_n}^{(n-1)}$ associated with the selected epoch is reinforced by $A_{in}$, that is $M_{\tau_n}^{(n)} = M_{\tau_n}^{(n-1)} + A_{in}.$ Other memories remain unchanged.

For $n>1$, the current step $\X_n$ is chosen as follows. If $\X_{\tau_n}$ was $\e_0$, then $\X_n$ becomes $\e_{I_n}$, that is, it is $\e_0$ with probability $1-p$, and, for $i=1, \dots, d$, it is $\e_i$ with probability $p/d$. If, on the other hand, $\X_{\tau_n}$ was $\e_i$ for some $i = 1, \ldots, d$, then $\X_n$ is $\e_{i J_n}$, that is, it takes values $\e_0$ and $\e_i$ with probability $q$ and $1-q$ respectively. Finally, the current epoch $n$ is assigned memory $1$, i.e., $M_n^{(n)} = 1$.

Theorem~\ref{thm: iid} and the connection between ERW and urn models from \cite{Bertoin} give the strong law behavior of $(\Loc_n)_{n\geq 0}$. Note that, it depends on the mixing parameters $p, q$, but not on the mean memory reinforcement parameter $a$.

\begin{theorem}
    Consider the delayed elephant random walk on the positive orthant in
    dimension $d$ with random reinforcement of memory, parametrized by $a$,
    $p$ and $q$ as described above. Then $\frac1n \Loc_n \to \frac{p}{d(p+q)} \1$ almost surely and in $L^1$.
\end{theorem}

\begin{proof}
It was noted in \cite{Bertoin} that the evolution of the ERW depends on
the moves at the selected past epoch rather than the selected epoch
itself.
Thus we consider an urn model with memory as objects categorized by the
types of moves.
Consider the vector of memory content of each type of step at epoch $n\geq 1$, denoted by $\boldsymbol{W}^{(n)}:=(W_0^{(n)}, \ldots, W_d^{(n)})$, where $W_i^{(n)}:=\sum_{1\leq k\leq n} M_k^{(n)}\indic{\X_k=\e_i}$ for $i = 0, 1, \ldots, d$. At epoch $n$, memory of the type of $\X_{\tau_n}$ is increased by the random reinforcement, while, memory of the type of $\X_n$ is increased by $1$. Thus $\Lp\boldsymbol{W}^{(n)}\Rp_{n\geq 0}$ behaves as an urn model of $(d+1)$ types (indexed by $0, 1, \ldots, d$) with i.i.d.\ replacement matrices having common mean 
\[
\R =
\begin{pmatrix}
  1-p+a & \frac{p}{d} & \frac{p}{d} & \cdots & \frac{p}{d}\\
  q & 1-q+a & 0 & \cdots & 0\\
  \vdots & \vdots & \ddots & \vdots & \vdots\\
  q & 0 & 0 & \cdots & 1-q+a
\end{pmatrix}.
\]
Clearly, the dominant eigenvalue of $\R$ is the common row sum $(1+a)$ and the corresponding left eigenvector normalized to probability is $\frac{1}{d(p+q)} \Lp dq, p, \ldots, p\Rp$.
Finally, note that, for $i = 1, \ldots, d$, the memory of the step at epoch $k$ can be reinforced by $1$ at epoch $k$ and, further in future epochs by a random amount if selected, and, hence, for $i=1,\ldots,d$
\begin{align*}
\frac1n W_i^{(n)} & = \frac1n \sum_{k=1}^{n} \Lp 1 + \sum_{l=k+1}^n A_{il} \indic{\tau_l = k} \Rp \indic{\X_k=\e_i}\\
 & = \frac1n L_{ni} + \frac{a}{n} \sum_{l=2}^n \indic{\X_{\tau_l}=\e_i} + \frac1n \sum_{l=2}^n (A_{il}-a) \indic{\X_{\tau_l}=\e_i}.\numberthis\label{eq: ERW1}
\end{align*}
By Theorem~\ref{thm: iid}, the composition vector $\frac1n {W_i^{(n)}} \to \frac{(1+a)p}{d(p+q)}$ and the count vector $\frac{1}{n}\sum_{l=2}^n \indic{\X_{\tau_l}=\e_i} \to \frac{p}{d(p+q)}$.
Finally, by Lemma~\ref{lemma: trunc mg}, the last term of \eqref{eq: ERW1} is negligible, with $(A_{il}+a)$ as the dominator. Hence the result follows.
\end{proof}

\section*{Acknowledgment}
The research of the second author was partly supported by MATRICS grant number MTR/2019/001448 from SERB, Govt.\ of India.

The authors also thank an anonymous referee for pointing out the unpublished manuscript~\cite{ZhangNew}.

\providecommand{\bysame}{\leavevmode\hbox to3em{\hrulefill}\thinspace}
\providecommand{\MR}{\relax\ifhmode\unskip\space\fi MR }
\providecommand{\MRhref}[2]{%
  \href{http://www.ams.org/mathscinet-getitem?mr=#1}{#2}
}
\providecommand{\href}[2]{#2}

\end{document}